\documentclass[11pt,a4paper]{article}

 \usepackage{amsmath}
  \usepackage{paralist}
  \usepackage{graphics} 
  \usepackage{epsfig} 

  \usepackage{hyperref} 

  \usepackage{multirow}
  \usepackage{amssymb} 

\newtheorem{theorem}{Theorem}[section]

\newtheorem{lemma}[theorem]{Lemma}

\newtheorem{remark}{Remark}

\newcommand{\F}{\mathbb{F}}
\newcommand{\N}{\mathbb{N}}
\newcommand{\Z}{\mathbb{Z}}

\newcommand{\IS}{\mathbb{S}}

\newcommand{\down}{\mathord{\downarrow}}
\DeclareMathOperator{\PGL}{PGL}
\DeclareMathOperator{\AGL}{AGL}
\DeclareMathOperator{\PGammaL}{P\Gamma L}
\DeclareMathOperator{\GammaL}{\Gamma L}
\DeclareMathOperator{\AGammaL}{A\Gamma L}
\DeclareMathOperator{\GL}{GL}
\DeclareMathOperator{\PHG}{PHG}
\DeclareMathOperator{\AHG}{AHG}
\DeclareMathOperator{\PG}{PG}
\DeclareMathOperator{\AG}{AG}
\DeclareMathOperator{\GR}{GR}
\DeclareMathOperator{\rad}{rad}
\DeclareMathOperator{\Aut}{Aut}
\DeclareMathOperator{\Inn}{Inn}

\newenvironment{proof}
{\begin{trivlist}\item[]{{\sc Proof.}}}{\hfill{$\square$}\noindent\end{trivlist}}

\begin{document}

\title{$2$-arcs of maximal size\\in the affine and the projective Hjelmslev plane\\over $\Z_{25}$}
\author{Michael Kiermaier, Matthias Koch and Sascha Kurz}




\maketitle

\begin{abstract}
It is shown that the maximal size of a $2$-arc in the projective Hjelmslev plane
over $\mathbb{Z}_{25}$ is $21$, and the $(21,2)$-arc is unique up to isomorphism.
Furthermore, all maximal $(20,2)$-arcs in the affine Hjelmslev plane over
$\mathbb{Z}_{25}$ are classified up to isomorphism.
\end{abstract}

\section{Introduction}
It is well known that a Desarguesian projective plane of order $q$ admits a $2$-arc of size $q + 2$ if and only if $q$ is even.
These $2$-arcs are called \emph{hyperovals}.
The biggest $2$-arcs in the Desarguesian projective planes of odd order $q$ have size $q + 1$ and are called \emph{ovals}.

For a projective Hjelmslev plane over a chain ring $R$ of composition length $2$ and size $q^2$, the situation is somewhat similar:
In $\PHG(2,R)$ there exists a hyperoval -- that is a $2$-arc of size $q^2 + q + 1$ -- if and only if $R$ is a Galois ring of size $q^2$ with $q$ even, see \cite{Honold-Landjev-2005-FFA11:292-304, Honold-Landjev-2001-DM231:265-278}
For the case $q$ odd and $R$ not a Galois ring it was recently shown that the maximum size $m_2(R)$ of a $2$-arc is $q^2$, see \cite{Honold-Kiermaier-maximal2Arcs-2010}.

In the remaining cases, the situation is less clear.
For even $q$ and $R$ not a Galois ring, it is known that the maximum possible size of a $2$-arc is lower bounded by $q^2 + 2$ \cite{Honold-Kiermaier-maximal2Arcs-2010} and upper bounded by $q^2 + q$ \cite{Honold-Landjev-2001-DM231:265-278}.
The last case $q$ odd and $R$ a Galois ring currently is the least satisfactory one.
Besides the upper bound $q^2$, only the lower bound $\left(\frac{q+1}{2}\right)^2$ is known~\cite{Honold-Kiermaier-maximal2Arcs-2010}, leaving a comparatively large gap.

The current state of the lower and upper bounds on the maximal size of a $2$-arc for general $q$ is summarized in the following table:
\[
    \begin{array}{|cc||c|c|}
	\hline
	\multicolumn{2}{|c||}{\multirow{2}{*}{$m_2(R)$}} & \multicolumn{2}{|c|}{R} \\
	\cline{3-4}
	& & \mbox{Galois ring} & \mbox{not a Galois ring}\\
	\hline\hline
	\multirow{2}{*}{$q$} & \multicolumn{1}{|c||}{\mbox{even}} & q^2 + q + 1 & q^2 + 2\leq\cdot\leq q^2 + q\\
	\cline{2-4}
	& \multicolumn{1}{|c||}{\mbox{odd}} & \left(\frac{q+1}{2}\right)^2\leq\cdot\leq q^2 & q^2\\
	\hline
    \end{array}
\]

In the open cases for $\#R\leq 16$ the exact values were determined computationally \cite{Kiermaier-2006, Honold-Kiermaier-2006-ProcACCT10:112-117}.
The ring $R = \Z_{25}$ is the smallest one where the exact value $m_2(R)$ was not known beforehand.
In \cite{Hemme-Weijand-1999} the first $2$-arc in $\PHG(2,\Z_{25})$ of size $20$ was found.
This was the biggest known $2$-arc until in \cite{Kiermaier-Koch-2009-ProcOC11:106-113} a $2$-arc of size $21$ was given.
The main result of this paper is:

\begin{theorem}
\label{thm:main}
The maximum size of a $2$-arc in $\PHG(2,\Z_{25})$ is $21$.
The $(21,2)$-arc is unique up to isomorphism.
\end{theorem}

The projective Hjelmslev plane over $\Z_{25}$ has $775$ points, so the search space for arcs of size $21$ or $22$ is way too big to allow a direct attack by a backtrack search.
Therefore, special computational methods had to be developed.

In Section~\ref{sect:prelim}, we give a brief introduction to finite chain rings and coordinate Hjelmslev planes.
For details, see for example \cite{Nečaev-1973-SbornM20[3]:364-382,Nechaev-FiniteRings-2008}\footnote{In these two references, finite chain rings are called \emph{Galois-Eisenstein-Ore}-rings (\emph{GEO}-rings).} and \cite{Honold-Landjev-2001-DM231:265-278}, respectively.
For the computation, two independent approaches are used:
In Section~\ref{sect:matthias}, in a first step the possible images of a $(22,2)$-arc in the factor plane are computed up to isomorphism, and in the second step the preimages are investigated.
In Section~\ref{sect:sascha} first the $(20,2)$-arcs in the affine plane are classified and then checked for extendability to a $(22,2)$-arc in the projective plane.
In Section~\ref{sect:further}, further computational results are summarized, and the maximum size of a $2$-arc of a given maximum point class multiplicity is determined for the projective Hjelmslev planes over $\Z_{25}$ and $\IS_5$, which is the other chain ring of composition length $2$ and size $25$.

Computational nonexistence- and uniqueness proofs using highly non-trivial algorithms are always a delicate matter, since already a subtle glitch in the implementation may very well cause a totally wrong result.
For that reason we put a lot of effort on assuring the correctness of our implementations.
The maximality and uniqueness of the $(21,2)$-arc was achieved independently by our two approaches.
Furthermore, both approaches independently report the number of the $\PGL(3,R)\down$-isomorphism classes of $(20,2)$-arcs in $\AHG(2,\Z_{25})$ as $488$.

\section{Notation and Preliminaries}
\label{sect:prelim}
\subsection{Finite chain rings}

Let $R$ be a finite ring.\footnote{Rings are assumed to contain an identity element $1\neq 0$ and to be associative, but not necessarily commutative.}
$R$ is called \emph{chain ring} if the lattice of the left-ideals is a chain.
A chain ring is necessarily local, so there is a unique maximum ideal $N = \rad(R)$.
It can be shown that $N$ is a principal ideal.
Because $R$ is finite, the quotient ring $R/N$ is isomorphic to a finite field $\F_q$ of order $q = p^r$ with $p$ prime.
$R/N$ is called the \emph{residue class field} of $R$.

We will need the projection $\phi : R\rightarrow \F_q$, $a\mapsto a\mod N$, which is a surjective ring homomorphism.
The number of ideals of $R$ reduced by $1$ is the composition length of $R$, considered as a left module ${}_R R$.
This number will be denoted by $m$.
The order of $R$ is $q^m$.

An important subclass of the finite chain rings are the \emph{Galois rings}. Their definition is a slight generalization of the construction of finite fields via irreducible polynomials:

Let $p$ be prime, $r$ and $m$ positive integers, $q = p^r$ and $f\in\Z_{p^m}[X]$ be a monic polynomial of degree $r$ such that the image of $f$ modulo $p$ is irreducible in $\F_p[X]$.
Then the \emph{Galois ring} of order $q^m$ and characteristic $p^m$ is defined as
\[
    \GR(q^m,p^m) = \Z_{p^m}[X]/(f).
\]
Up to isomorphism, the definition is independent of the particular choice of $f$.
The symbols $p$, $q$, $r$ and $m$ are consistent with the earlier definitions: $\GR(q^m,p^m)$ $/$ $\rad(\GR(q^m,p^m)) \cong \F_q$ and the composition length of $\GR(q^m,p^m)$ is $m$.
Furthermore, the class of the Galois rings contains the finite fields and the integer residue class rings modulo a prime power:
$\GR(p^m,p^m)\cong\Z_{p^m}$ and $\GR(p^r,p)\cong \F_{p^r}$.

The finite fields are exactly the finite chain rings of composition length $1$.
In the following we assume that $R$ is a finite chain ring of composition length $2$.
The isomorphism types of these rings are known\cite{Raghavendran-1969-ComposM21[2]:195-229, Nečaev-1973-SbornM20[3]:364-382, Cronheim-1978-GeomDed7:287-302}:
For a fixed size $q = p^r$ of the residue field of $R$ there are $r + 1$ possible isomorphism types for $R$.
One is the Galois ring $\GR(q^2,p^2)$ of characteristic $p^2$, and the remaining $r$ ones are all of characteristic $p$ and not isomorphic to a Galois ring.
Among the latter $r$ possibilities there is a single commutative ring, which is $\IS_q = \F_q[X]/(X^2)$.

In particular, up to isomorphism the only chain rings of size $25$ and composition length $2$ are $\Z_{25}$ and $\IS_5 = \F_{5}[X]/(X^2)$.

\subsection{Affine and projective coordinate Hjelmslev planes}
\label{sect:geometry}

The \emph{projective Hjelmslev plane} $\PHG(2,R)$ over a finite chain ring $R$ is defined as follows:
The point set $\mathcal{P}(\PHG(2,R))$ [line set $\mathcal{L}(\PHG(2,R))$] is the set of the free rank $1$ [rank $2$] right submodules of the right $R$-module $R^3$, and the incidence relation is given by set inclusion.
For a point $x$ a vector $\mathbf{v}\in R^3$ with $x = \mathbf{v}R$ is called \emph{coordinate vector} of $x$.
$x$ has $\#R^\times = q(q-1)$ coordinate vectors.
There is a unique coordinate vector of $x$ whose first unit entry is $1$.

The \emph{affine Hjelmslev plane} $\AHG(2,R)$ over a finite chain ring $R$ is defined as follows:
The point set $\mathcal{P}(\AHG(2,R))$ is the set of the vectors in $R^2$, and the line set $\mathcal{L}(\AHG(2,R))$ is the set of all cosets of free rank $1$ submodules of the right $R$-module $R^2$.
The incidence relation is given by set membership.

For the number of points and lines we have
\[
\#\mathcal{P}(\PHG(2,R)) = \#\mathcal{L}(\PHG(2,R)) = \frac{q^3 - 1}{q-1} q^2\quad\mbox{and}
\]
\[
\#\mathcal{P}(\AHG(2,R)) = q^4, \quad \#\mathcal{L}(\AHG(2,R)) = q^3 (q+1).
\]

The projective and affine Hjelmslev planes share a lot of structure which will be described simultaneously: Let $(\mathcal{H},\mathcal{G})$ be either $(\PHG(2,R),\PG(2,\F_q))$ or $(\AHG(2,R)$ $,\AG(2,\F_q))$.
For a consistent description, in the following we identify a line $L$ of $\PHG(2,R)$ with the set of points incident with $L$, so that the incidence relation of both $\PHG(2,R)$ and $\AHG(2,R)$ is given by set membership.

In contrast to classical planes, in Hjelmslev planes it may happen that two distinct lines meet each other in more than a single point.
More precisely, there is more than one line passing through points $x$ and $y$ of $\mathcal{H}$ if and only if $\phi(x) = \phi(y)$, where the mapping $\phi$ is extended from $R$ to $\mathcal{P}(\mathcal{H})$.
Each preimage $\phi^{-1}(z)$ with $z\in\mathcal{P}(\mathcal{G})$ is called \emph{point class}.
For $x\in\mathcal{P}(\mathcal{H})$, the point class $\phi^{-1}(\phi(x))$ containing $x$ is denoted by $[x]$.
Similarly, two lines $L_1$ and $L_2$ of $\mathcal{H}$ intersect in more than one point if and only if $\phi(L_1) = \phi(L_2)$.
The preimages $\phi^{-1}(l)$ with $l\in\mathcal{L}(\mathcal{G})$ are called \emph{line classes}.
For $L\in\mathcal{L}(\mathcal{H})$, the line class $\phi^{-1}(\phi(L))$ containing $L$ is denoted by $[L]$.
So by our definition, a line class $[L]$ is a set of points.

Similarly to the classical case, the projective and the affine Hjelmslev plane over $R$ are tightly related:
By the mapping $\mathcal{P}(\AHG(2,R))\rightarrow\mathcal{P}(\PHG(2,R))$, $\mathbf{v}\mapsto (1,\mathbf{v})R$, the affine Hjelmslev plane is embedded in the projective Hjelmslev plane.
We will refer to this embedding as the \emph{standard embedding}.
The points in $\PHG(2,R)$ not contained in the image of $\AHG(2,R)$ form a line class.
This line class is called \emph{line class at infinity} and given by all points $\mathbf{v}R$ such that the first component of $\mathbf{v}$ is not a unit.
On the other hand, removing a line class $[L]$ from $\PHG(2,R)$ yields a plane isomorphic to $\AHG(2,R)$, and $[L]$ is its line class at infinity.

There is also an axiomatic definition of projective and affine Hjelmslev planes which does not rely on an underlying chain ring \cite{Klingenberg-1954-MathZeit60:384-406, Lüneburg-1962-MathZeit79:260-288}.
This definition is more general than the one given above: Not every axiomatically defined Hjelmslev plane can be coordinatized over a finite chain ring.
For that reason, the Hjelmslev planes $\PHG(2,R)$ and $\AHG(2,R)$ are also called \emph{coordinate} Hjelmslev planes.
The axiomatic definition of an affine Hjelmslev plane involves a \emph{parallelism}, which is an equivalence relation on the set of lines satisfying Euclid's parallel axiom.
In contrast to classical affine planes, in general the parallelism of an affine Hjelmslev plane is not uniquely determined by its underlying incidence structure.
For a coordinate affine Hjelmslev plane $\AHG(2,R)$ embedded in $\PHG(2,R)$, each line $L$ contained in the line class at infinity induces a parallelism in the following way: Two lines in $\mathcal{L}(\AHG(2,R))$ are called parallel if and only if they pass through the same point on $L$.
Using the standard embedding, we will call the parallelism induced by the line $(1,0,0)^\perp$ the \emph{standard parallelism}.

\subsection{Induced subgeometries}
Because $\phi : \mathcal{P}(\mathcal{H})\rightarrow\mathcal{P}(\mathcal{G})$ maps lines to lines, the geometry of the point classes and the line classes with incidence given by the subset relation is isomorphic to the plane $\mathcal{G}$.
This plane is called \emph{factor plane} of $\mathcal{H}$.

The restriction of the geometry $\mathcal{H}$ to a single point class $[x]$ is isomorphic to the affine plane $\AG(2,\F_q)$.
It will be denoted by $\Pi_{[x]}$.
Each line in one of these affine geometries is called \emph{line segment} when considered as a subset of $\mathcal{P}(\mathcal{H})$.

Now let $L\in\mathcal{L}(\mathcal{H})$ be a line.
Let $[x]$ be a point class incident with $[L]$.
Then $L\cap[x]$ is a line segment contained in $[x]$.
The set of all line segments $P$ arising as $L^\prime\cap[x]$ with $L^\prime\in[L]$ forms a parallel class of lines in the affine plane on $[x]$.
In this way, the parallel classes on $[x]$ are in bijection with the line classes incident with $[x]$ in the factor plane.
We will say that the line class $[L]$ \emph{determines} the parallel class $P$, and that the parallel class $P$ \emph{determines} the line class $[L]$.
Since a parallel class in $\Pi_{[x]}$ is already determined by a single line or by the unique line passing through a pair of distinct points in $[x]$, we will also use the formulation that the line class $[L]$ is determined by a line segment or by a pair of distinct points.

In the projective case $\mathcal{H} = \PHG(2,R)$ the line segments determined by a line class form a geometry in the following sense:
Let
\[\mathfrak{P} = \{\mbox{all line segments determined by }[L]\}\cup\{\infty\}.\]
For each point class $[x]$ incident with $[L]$ we define
\[L_{[x]} = \{\mbox{all line segments determined by }[L]\mbox{ contained in }[x]\}\cup\{\infty\}
\]
and $\mathfrak{L} = \{L' : L'\in\mathcal{L}(\mathcal{H}), L'\subset [L]\} \cup \{L_{[x]} : x \in [L]\}$.
Then the geometry $\Pi_{[L]}$ consisting of the point set $\mathfrak{P}$ and the line set $\mathfrak{L}$ with the incidence relation defined in the obvious way\footnote{Let $X\in\mathfrak{L}$. A line segment $S$ is incident with $X$ if $S\subseteq L$, and the point $\infty$ is incident with $X$ if $\infty\in L$.} is isomorphic to the projective plane $\PG(2,\F_q)$.

\subsection{Collineations}
\label{subsect:collineations}
A \emph{collineation} of a point-line geometry is a bijection on the point set mapping lines to lines.
As usual, two multisets of points are called \emph{isomorphic} if they are contained in the same orbit under the action of the collineation group.
In this article we do not require that a collineation of an affine Hjelmslev plane preserves its parallelism.
The reason for this is that the group of collineations will be used to reduce the search space of the maximal arc problem, which does not depend on the parallelism.

The group $\GammaL(3,R)$ consists of all elements $(A,\sigma)$ with $A\in\GL(3,R)$ and $\sigma\in\Aut(R)$.
The multiplication in $\GammaL(3,R)$ is given by $(A,\sigma)\cdot(B,\rho) = (A\sigma(B),\sigma\rho)$.
Thus $\GammaL(3,R)$ is a semidirect product $\GL(3,R)\rtimes\Aut(R)$.
$\GammaL(3,R)$ acts on the set of points of $\mathcal{H}$ via $(\mathbf{A},\sigma)\cdot \mathbf{v}R = \mathbf{A}\sigma(\mathbf{v}) R$.
The kernel of this group action is $N = \{(\mathbf{I}_3 \lambda^{-1}, a\mapsto \lambda a\lambda^{-1}) : \lambda\in R^*\}$, where $\mathbf{I}_3$ denotes the $3\times 3$ unit matrix.
Now the \emph{projective semilinear group} $\PGammaL(3,R)$ is defined as the factor group $\GammaL(3,R) / N$ and acts faithfully on $\mathcal{P}(\mathcal{H})$.
By the fundamental theorem of projective Hjelmslev geometry \cite{Kreuzer-1988}, $\PGammaL(3,R)$ is the group of collineations of $\PHG(2,R)$.

Starting with $\GL(3,R)$ instead of $\GammaL(3,R)$, the kernel of the group action on $\mathcal{P}(\mathcal{H})$ is $N = \mathbf{I}_3 Z(R^*)$, where $Z(R^*)$ is the center of the unit group of $R$.
The \emph{projective linear group} $\PGL(3,R)$ is defined as $\PG(3,R) / N$.
It is worth mentioning that while the index of $\GL(3,R)$ in $\GammaL(3,R)$ is $\#\Aut(R)$, the index of $\PGL(3,R)$ in $\PGammaL(3,R)$ is only $[\Aut(R) : \Inn(R)]$, where $\Inn(R)$ denotes the group of inner automorphisms of $R$.
$\PGL(3,R)$ acts transitively on the set of ordered quadruples of points in general position (no $2$ points in the same point class, no $3$ collinear non-empty point classes) \cite[Thm. 17]{Cronheim-1978-GeomDed7:287-302}.
This group action is regular if and only if $R$ is commutative.

In the following it is assumed that $\AHG(2,R)$ is embedded into $\PHG(2,R)$ by the standard embedding.
It is clear that all $\sigma\in\PGammaL(3,R)$ fixing $\AHG(2,R)$ as a set give rise to a collineation of $\AHG(2,R)$, which will be denoted by $\sigma\down$.
The stabilizer of $\AHG(2,R)$ in $\PGammaL(3,R)$ will be denoted by $\PGammaL(3,R)_\infty$, and the set of induced collineations of $\AHG(2,R)$ will be denoted by $\PGammaL(3,R)\down$.
The symbols $\PGL(3,R)_\infty$ and $\PGL(3,R)\down$ are defined analogously.
It can be checked that the restriction $\down\colon\PGammaL(3,R)_\infty\rightarrow\PGammaL(3,R)\down$ is a bijection.
In general, $\PGammaL(3,R)\down$ is not the full collineation group of $\AHG(2,R)$:
For example $\PGammaL(3,\Z_9)\down = \PGL(3,\Z_9)\down$ has index $3$ in the full collineation group of $\AHG(2,\Z_9)$.

Let $\rho : \mathbf{v}R\mapsto \mathbf{A}\sigma(\mathbf{v})R$ be an element of $\PGammaL(3,R)_\infty$, where $(\mathbf{A},\sigma)\in\GammaL(3,R)$ with $\mathbf{A}\in\GL(3,R)$ and $\sigma\in\Aut(R)$ is a representative of an element of $\PGammaL(3,R)$.
Then we may choose $\mathbf{A}$ in the form
\[
\mathbf{A} = \begin{pmatrix}1 & c_1 & c_2\\b_1 & a_1 & a_2\\b_2 & a_3 & a_4\end{pmatrix}
\]
with $a_1,a_2,a_3,a_4,b_1,b_2\in R$ and $c_1,c_2\in\rad(R)$.
Let $\mathbf{A}^\prime = \begin{pmatrix}a_1 & a_2\\a_3 & a_4\end{pmatrix}$, $\mathbf{b} = \begin{pmatrix}b_1 \\ b_2\end{pmatrix}$ and $\mathbf{c} = (c_1, c_2)\in \rad(R)\times \rad(R)$.
Since $\mathbf{A}$ is invertible, so is $\mathbf{A}^\prime$.

By the standard embedding each element of $\AHG(2,R)$ has the form $\begin{pmatrix}1\\\mathbf{v}\end{pmatrix}$ with $\mathbf{v}\in R^2$.
It holds $\rho\left(\begin{pmatrix}1\\\mathbf{v}\end{pmatrix}\right) = \mathbf{A} \begin{pmatrix}1\\\sigma(\mathbf{v})\end{pmatrix} = \begin{pmatrix}1 + \mathbf{c}\sigma(\mathbf{v})\\\mathbf{A}^\prime\sigma(\mathbf{v}) + \mathbf{b}\end{pmatrix}$.
Because of $\mathbf{c}\in\rad(R)\times\rad(R)$, $(1 + \mathbf{c}\sigma(\mathbf{v}))^{-1} = (1-\mathbf{c}\sigma(\mathbf{v}))$, so $\rho\left(\begin{pmatrix}1\\\mathbf{v}\end{pmatrix}\right)R = \begin{pmatrix}1 \\ (\mathbf{A}^\prime \sigma(\mathbf{v}) + \mathbf{b})(1 - \mathbf{c}\sigma(\mathbf{v}))\end{pmatrix}R$.
So in affine coordinates $\rho\down$ maps $\mathbf{v}$ to $-\mathbf{A}^\prime\sigma(\mathbf{v})\mathbf{c}\sigma(\mathbf{v}) + (\mathbf{A}^\prime - \mathbf{b}\mathbf{c})\sigma(\mathbf{v}) + \mathbf{b}$.

$\rho\down$ preserves the standard parallelism if and only if $\mathbf{c} = \mathbf{0}$.
Hence the mappings in $\PGammaL(3,R)\down$ preserving the standard parallelism are given by the group $\AGammaL(2,R)$ consisting of all mappings $R^2\rightarrow R^2, \mathbf{v}\mapsto \mathbf{A}\sigma(\mathbf{v}) + \mathbf{b}$ with $\mathbf{A}\in\GL(2,R)$, $\sigma\in\Aut(R)$ and $\mathbf{b}\in R^2$.
It holds
\begin{align*}
\#\AGammaL(2,R) & = (q-1)^2 q^9 (q+1)\cdot\#\Aut(R),\\
[\PGammaL(3,R)\down : \AGammaL(2,R)] & = q^2\quad\mbox{and}\\
[\PGammaL(3,R) : \PGammaL(3,R)\down] & = q^2 + q + 1.
\end{align*}

For example the ring of our main interest $R = \Z_{25}$ has only the trivial automorphism.
So $\PGammaL(3,\Z_{25}) = \PGL(3,\Z_{25})$, $\AGammaL(2,\Z_{25}) = \AGL(2,\Z_{25})$, $\#\PGL(3,\Z_{25}) = 145312500000$, $\#\PGL(3,\Z_{25})\down = 4687500000$ and $\#\AGL(2,\Z_{25}) = 187500000$.
We used the graph isomorphism package \texttt{nauty} \cite{McKay-Nauty} to see that $\PGL(3,\Z_{25})\down$ is in fact the full collineation group of $\AHG(2,\Z_{25})$.

\subsection{Arcs}
For any geometry and any $n\in\N$, a multiset of points $\mathfrak{k}$ of size $n$ is called \emph{$(n,w)$-arc}, if no $w+1$ elements of $\mathfrak{k}$ are collinear\footnote{Of course we have to
respect multiplicities for counting the number of collinear points.}.
We denote by $m_w(R)$ the maximum size of a $w$-arc in the projective Hjelmslev plane $\PHG(2,R)$.
There is an online table \cite{arc-tables} for known lower bounds on $m_w(R)$.

We want to mention that the literature is inconsistent about allowing an arc to be a proper multiset.
Our definition is motivated by the connection to coding theory \cite{Honold-Landjev-2000-EJC7:R11} where the possibly repeated columns of a generator matrix are interpreted as coordinate vectors in projective Hjelmslev geometries.
This article is about $2$-arcs, where the different possible definitions do not matter:
The only proper multiset which is a $2$-arc consists of a single point of multiplicity $2$.

For an arc $\mathfrak{K}$ and a point set $X\subseteq\mathcal{P}$, the size $\#(X\cap\mathfrak{K})$ is called \emph{multiplicity} of $X$.
In this way, multiplicities are declared for point classes and line classes.
A point class of multiplicity $u$ will simply be called \emph{$u$-class}.
Furthermore, the maximum value of $u$ such that there is a $u$-class will be denoted by $u(\mathfrak{K})$.
It is the \emph{maximum point class multiplicity} of $\mathfrak{K}$.

\section{Computation via the images in the factor plane}
\label{sect:matthias}

We define the map $\Phi : \N^{\mathcal{P}(\PHG(2,R))} \rightarrow \N^{\mathcal{P}(\PG(2,R))}$ as the extension of $\phi$ to multisets.
For a multiset of points $\mathfrak{K}$ in $\PHG(2,R)$, the image $\Phi(\mathfrak{K})$ reduces $\mathfrak{K}$ to the distribution of the points in $\mathfrak{K}$ into the point classes of $\PHG(2,R)$.

The strategy for the algorithm is to generate all possible images
$\Phi(\mathfrak{K})$ up to $\PG(2,\F_q)$-isomorphism in a first step.
In a second step, the preimages of these images are generated.
Concerning the algorithmic complexity, the second step is much harder than the first
one.

This approach suggests itself for several reasons:
Without the intermediate step, the computations must be done under the full collineation group $\PGammaL(3,R)$.
But to compute the possible preimages of a given point class distribution $\mathfrak{k}$ in the second step, the acting group is reduced to the preimage in $\PGammaL(3,R)$ of the stabilizer of $\mathfrak{k}$ by the Homomorphism Principle~\cite{Laue-2001} and the fact that the mapping $\Phi$ naturally extends to a homomorphism of the group actions, mapping the action of $\PGammaL(3,R)$ on $\mathcal{P}(\PHG(2,R)$ to the action of $\PGammaL(3,\F_q)$ on $\mathcal{P}(\PG(2,\F_q))$, 
Furthermore, the next lemma (compare \cite[Theorem 3.5]{Honold-Landjev-2001-DM231:265-278}) summarizes severe restrictions on the possible point class distributions, which allow a drastic pruning of the search tree between the two steps:

\begin{lemma}
\label{lma:ncl_cond}
Let $\mathcal{H}$ be the projective Hjelmslev plane over $R$ and assume that its factor plane $\mathcal{G}$ is of odd order $q$.
Let $\mathfrak{K}$ be a $2$-arc in $\mathcal{H}$.
\begin{enumerate}[(a)]
\item\label{lma:ncl_cond:ncl2} Each $2$-class is incident with a line class $[L]$ of multiplicity $2$.
\item\label{lma:ncl_cond:bs} Each line class is incident with a point class of multiplicity $0$.
\item\label{lma:ncl_cond:u2} If $u(\mathfrak{K}) = 2$, then $\#\mathfrak{K}\leq q^2 - q + 2$.
\item\label{lma:ncl_cond:ncl3} If $u(\mathfrak{K}) \geq 3$, then $\#\mathfrak{K} \le q^2 - 2q + 3$.
\item\label{lma:ncl_cond:q+2} If $u(\mathfrak{K})\leq 2$, then each line class has multiplicity at most $q+1$.
\item\label{lma:ncl_cond:q+1} If $u(\mathfrak{K})\leq 2$, then the number of $2$-classes incident with a line class of multiplicity $q+1$ is either $\frac{q-1}{2}$ or $\frac{q+1}{2}$.
\end{enumerate}
\end{lemma}

\begin{proof}
To show part~(\ref{lma:ncl_cond:ncl2}), let $[x]$ be a $2$-class.
Because $\mathfrak{K}$ is a $2$-arc, the line class determined by the $2$ points in $[x]$ has multiplicity $2$.

Let $[L]$ be a line class in $\mathcal{H}$.
For each $u$-class $[x]$ on $[L]$ of multiplicity $u \geq 3$ we arbitrarily remove $u-2$ points in $[x]$ from $\mathfrak{K}$.
For the remaining point set $\mathfrak{k}$ we have $\#(\mathfrak{k}\cap[x])\leq 2$ for all point classes incident with $[L]$.
In the case $u(\mathfrak{K})\leq 2$ it holds $\mathfrak{K} = \mathfrak{k}$.
The multiset of points in $\Pi_{[L]}$ induced by $\mathfrak{k}\cap [L]$ is a $2$-arc.
Because $q$ is odd, we get $\#(\mathfrak{k}\cap [L])\leq q+1$ which shows
part~(\ref{lma:ncl_cond:q+2}).
Furthermore, in the case $\#(\mathfrak{k}\cap [L]) = q+1$ the induced multiset of
points forms an oval in $\Pi_{[L]}$.
The point $\infty$ is either external or internal to this oval, so the lines
$L_{[x]} : x\in [L]$ through $\infty$ comprise either $\frac{q-1}{2}$
passants, $2$ tangents, $\frac{q-1}{2}$ secants or $\frac{q+1}{2}$ passants,
$\frac{q+1}{2}$ secants.
Since the line segments on $L_{[x]}$ cover the point class $[x]$,
parts~(\ref{lma:ncl_cond:bs}) and (\ref{lma:ncl_cond:q+1}) follow.

Now let $[x]$ be a point class of maximum multiplicity $u(\mathfrak{K})$.
$\mathfrak{K}\cap[x]$ is a $2$-arc in $\Pi_{[x]}$, so $u(\mathfrak{K}) = \#(\mathfrak{K}\cap[x]) \leq q+1$.
If $u(\mathfrak{K}) = 2$, then by part~(\ref{lma:ncl_cond:ncl2}) there is a line class $[L]$ through $[x]$ with $\#([L]\setminus[x]) = 0$ and by part~(\ref{lma:ncl_cond:q+2}) each of the remaining $q$ line classes through $[x]$ contain at most $q-1$ points outside of $[x]$.
This shows $\#\mathfrak{K} \leq 2 + q(q-1) = q^2 - q + 2$, which is part~(\ref{lma:ncl_cond:u2}).

If $u(\mathfrak{K}) = 3$, then the pairs of points in $\mathfrak{K}\cap[x]$ determine $3$ line
classes.
Since $\mathfrak{K}$ is a $2$-arc, these line classes cannot contain any other point of $\mathfrak{K}$.
Now let $[L]$ be one of the $q-2$ remaining line classes incident with $[x]$, and $y\in\mathfrak{K}\cap [x]$.
There are $q$ lines in $[L]$ incident with $y$, and these lines cover $[L]\setminus
[x]$.
On each of these lines there is at most one further point of $\mathfrak{K}$.
So we get $\#\mathfrak{K} \leq 3 + q(q-2) = q^2 - 2q + 3$.
Since $q$ is odd, for $u(\mathfrak{K}) \ge 4$ the pairs of points in $\mathfrak{K}\cap[x]$ determine
at least $4$ line classes.
The same reasoning leads to $\#\mathfrak{K} \leq u + q(q-3) \leq (q+1) + q(q-3) =
q^2 - 2q + 1$.
This shows part~(\ref{lma:ncl_cond:ncl3}).
\end{proof}

\begin{remark}\label{rmk:ubig}
For odd $q\geq 5$ each point of $\PG(2,q)$ lies on at least $2$ secants of a given oval.
This shows that in the two cases of Lemma~\ref{lma:ncl_cond}(\ref{lma:ncl_cond:q+1}), the point set $\mathfrak{K}\cap[L]$ cannot be extended if we drop the condition $u(\mathfrak{K})\leq 2$.
This puts serious restrictions on the point class distribution of a line class $[L]$ of multiplicity $>q+1$, which can be used to improve the bound in part~(\ref{lma:ncl_cond:ncl3}) of the preceding Lemma.

In Theorem~\ref{thm:u}, this will be done explicitly for the chain rings $\Z_{25}$ and $\IS_5$.
\end{remark}

\subsection{Creating the point class distributions}
The first task for constructing all $(n,2)$-arcs in $\PHG(2,\Z_{25})$ is the
creation of all isomorphism types of point class distributions obeying the
restrictions of Lemma~\ref{lma:ncl_cond}.
For the critical case $n=22$, this can be done without the use of a computer:

\begin{lemma}
\label{lma:22_2_class_types}
Let $\mathfrak{K}$ be a $(22,2)$-arc in the projective Hjelmslev plane over $\Z_{25}$ or $\IS_5$.
The isomorphism type of the point class distribution $\Phi(\mathfrak{K})$
in the factor plane is given by one of the following cases:
\begin{enumerate}[(1)]
\item\label{lma:22_2_class_types:1} $u(\mathfrak{K}) = 1$.
\begin{enumerate}[(a)]
\item\label{lma:22_2_class_types:1:a} In the factor plane, the $0$-classes are given by a line $L$ and three collinear points not incident with $L$.
\item\label{lma:22_2_class_types:1:b} In the factor plane, the $0$-classes classes are given by a line $L$ and three non-collinear points not incident with $L$.
\item\label{lma:22_2_class_types:1:c} The $0$-classes form a projective triangle in the factor plane.
\end{enumerate}
\item\label{lma:22_2_class_types:2} $u(\mathfrak{K}) = 2$, and there are six $2$-classes forming an oval in the factor plane.
The point pairs in the $2$-classes determine the line classes tangent to the oval.
The point classes internal to the oval are $1$-classes, and the point classes external to the oval are $0$-classes.
\end{enumerate}
\end{lemma}

\begin{proof}
By Lemma~\ref{lma:ncl_cond}(\ref{lma:ncl_cond:ncl3}), $u(\mathfrak{K})\leq 2$.
If $u(\mathfrak{K}) = 1$, by Lemma~\ref{lma:ncl_cond}(\ref{lma:ncl_cond:bs}) the $0$-classes form a blocking set of size $31 - 22 = 9$ in the factor plane.
It is known that such a blocking set either contains a line, or it is isomorphic to
a projective triangle.
Using the fact that the collineation group of $\PHG(2,R)$ acts transitively on the quadruples of points in general position, this gives the $3$ isomorphism types listed in case~(\ref{lma:22_2_class_types:1}).

Now we assume $u(\mathfrak{K}) = 2$.
Let $P_i$ be the set of all $i$-classes and $[x]\in P_2$ be a $2$-class.
By Lemma~\ref{lma:ncl_cond}(\ref{lma:ncl_cond:ncl2}), one of the line classes incident with $[x]$ is of multiplicity~$2$, and by Lemma~\ref{lma:ncl_cond}(\ref{lma:ncl_cond:q+2}) the remaining $5$ line
classes incident with $[x]$ contain at most $6$ points of $\mathfrak{K}$ each.
Because of $\#\mathfrak{K} = 22$, it follows that these $5$ line classes
contain exactly $6$ points of $\mathfrak{K}$ each, so by Lemma~\ref{lma:ncl_cond}(\ref{lma:ncl_cond:q+1}) each of these line classes is incident with one or two further $2$-classes.
This shows $6\leq \#P_2 \leq 11$.

Let $L_i$ denote the set of line classes incident with exactly $i$ point
classes in $P_2$.
We have seen that each point class $[y]\in P_2$ is incident with exactly
$5$ line classes in $L_2\cup L_3$, so double counting the set of flags
$([y],[L])$ with $[y]\in P_2$, $[L]$ incident with $[y]$ and $[L]\in L_2\cup L_3$
gives the equation $5\#P_2 = 2\#L_2 + 3\#L_3$.
Furthermore, double counting the pairs of point classes in $P_2$ shows $\binom{\#P_2}{2} =
\#L_2 + 3\#L_3$.
Solving this equation system for $(\#L_2,\#L_3)$ yields $2\#L_2 = \#P_2 (11 -
\#P_2)$ and $3\#L_3 = \#P_2 (\#P_2 - 6)$.
So $3\mid \#P_2$ and therefore $\#P_2\in\{6,9\}$.

In the case $\#P_2 = 9$ it holds $\#P_1 = 4$ and $\#L_2 = \#L_3 = 9$.
Let $X$ be the set of all flags $([y],[L])$ with $[y]\in P_1$, $[L]$ incident with
$[y]$ and $[L]\in L_2$.
We count $\#X$ in two ways:
Since all line classes in $L_2$ have multiplicity $6$, each line class in $L_2$ is incident with exactly $2$ point classes in $P_1$.
This shows $\#X = 2\#L_2 = 18$.
On the other hand, because of $\#P_2 = 9$, a point class $[y]\in P_1$ is
incident with at most $4$ line classes in $L_2$, which leads to the
contradiction $\#X \leq 4\#P_1 = 16$.

So we get $\#P_2 = 6$ and $\#L_3 = 0$.
This means that the point classes in $P_2$ form an oval in the factor plane.
The only possibility for the line classes determined by the point pairs
$\mathfrak{K}\cap [y]$, $[y]\in P_2$ are the line classes tangent to the
oval $P_2$.
It follows that the $15$ point classes external to $P_2$ are $0$-classes.
Because of $\#\mathfrak{K} = 22$, all the point classes internal to $P_2$ must be $1$-classes.
\end{proof}

\begin{remark}
\begin{enumerate}[(a)]
\item For the base ring $\Z_{25}$, the point class distributions
(\ref{lma:22_2_class_types:1:a}) and (\ref{lma:22_2_class_types:1:b}) are not
possible:
Because there is an empty line class, such a $(22,2)$-arc would be
contained in an affine subplane of $\PHG(2,\Z_{25})$.
But in \cite{Kurz-2009-SJoC3:159-178} it was shown computationally that the maximum
size of a $2$-arc in $\AHG(2,\Z_{25})$ is $20$.
\item While our main interest is in the projective Hjelmslev plane over $\Z_{25}$, Lemma~\ref{lma:22_2_class_types} also holds for $\IS_5$.
Let $\mathcal{H} = \PHG(2,\IS_5)$.
The existence of a $(25,2)$-arc in $\mathcal{H}$ implies the existence of a
$(22,2)$-arc in $\mathcal{H}$ realizing the point class distributions
(\ref{lma:22_2_class_types:1:a}) and (\ref{lma:22_2_class_types:1:b}).
Furthermore, the $(22,2)$-arc $\mathfrak{k}_{\IS_5}$ in $\mathcal{H}$ given in
\cite{Kiermaier-Koch-2009-ProcOC11:106-113} realizes the point class
distribution (\ref{lma:22_2_class_types:2}).
\end{enumerate}
\end{remark}

The automatic construction of the point class distribution is done by a
backtracking search, checking the conditions of Lemma~\ref{lma:ncl_cond} in each
step.
Because only one representative of each isomorphism class is needed, we apply
orderly generation \cite{0392.05001}.
Whenever a complete point class distribution $\mathfrak{k}$ is found, all
$2$-arcs in its preimage in $\PHG(2,\Z_{25})$ are created as described below.
If in this search it turns out that even a smaller\footnote{The common partial order
on multisets is used.} point class distribution $\mathfrak{k}' < \mathfrak{k}$
does not admit an $(n,2)$-arc in its preimage, it is clear that there does not exist
an $(n,2)$-arc $\mathfrak{K}$ with $\Phi(\mathfrak{K}) = \mathfrak{k}$.
Thus $\mathfrak{k}'$, or even stronger all isomorphic copies of $\mathfrak{k}'$, are
a \emph{forbidden substructure} for the point class distributions of an
$(n,2)$-arc.
This restriction on the point class distributions is not covered by
Lemma~\ref{lma:ncl_cond} and can be used additionally for the ongoing construction
of point class distributions.

\subsection{Lifting point class distributions to $2$-arcs in $\PHG(2,\Z_{25})$}
After the generation of a possible point class distribution, the preimages in the
Hjelmslev plane are built.
Starting from $\mathfrak{K} = \emptyset$ we iteratively extend $\mathfrak{K}$ point-wise, with respect to the point
class distribution $\mathfrak{k}$ and the $2$-arc property.
To make our computation more efficient we construct the possible point sets only up
to isomorphism.
For the isomorphism test we use a combination of orderly generation \cite{0392.05001} and the ladder game, based on the homomorphism principle \cite{Laue-1993-BayMS43:53-96,Schmalz-1993-JCD1[2]:125-170,Laue-2001}.

The orderly generation algorithm needs an assignment of the predicate
\textit{canonical} to exactly one representative in each orbit.
The two-step approach also works for the computation of a canonical representative:
Having such a predicate in the factor plane $\PG(2,\Z_5)$ together with the homomorphism of group actions $\Phi$, by the Homomorphism Principle we are able to arrange the definition of canonicity in $\N^{\mathcal{P}(\PHG(2,\Z_{25}))}$ in such a way that the image of a canonical point set under the mapping $ \Phi$ is always canonical.
Furthermore by the Homomorphism Principle the search space for the canonical candidate can be reduced:
For a point set $\mathfrak{K}$ in the preimage of a canonical point class distribution $\mathfrak{k}$, the canonical form of $\mathfrak{K}$ must be in the orbit of $\mathfrak{K}$ under the $\Phi$-preimage of the stabilizer of $\mathfrak{k}$ in $\PGammaL(3,\Z_5)$.
This group usually is much smaller than the full group $\PGammaL(3,\Z_{25})$.

These methods were implemented in C++ and executed on a single CPU of type Intel Xeon E5520.
Running the program for $(22,2)$-arcs in $\PHG(2,\Z_{25})$ took
8.5 hours and gave no result.
Running the program for $(21,2)$-arcs took 13.5 hours and returned exactly one $(21,2)$-arc, the arc already
found in \cite{Kiermaier-Koch-2009-ProcOC11:106-113}.
This proves Theorem~\ref{thm:main}.

\section{Computation via the extension of $2$-arcs in the affine plane}
\label{sect:sascha}
Now we follow a different approach:
First the $(20,2)$- and $(19,2)$-arcs in the affine plane $\AHG(2,\Z_{25})$ are classified and then tested for extendability to a $(22,2)$- respectively $(21,2)$-arc in $\PHG(2,\Z_{25})$.
This is justified by the following lemma:

\begin{lemma}
\label{lma:projToAff}
Let $\mathfrak{K}$ be a $2$-arc in the projective Hjelmslev plane over $\Z_{25}$ or $\IS_5$ with $u(\mathfrak{K})\leq 2$.
Then there is a line class of multiplicity at most $2$.
Thus by removing this line class, it remains a $2$-arc in $\AHG(2,R)$ of size at least $\#\mathfrak{K}-2$.
\end{lemma}

\begin{proof}
If there exists a $2$-class, a suitable line class is given by Lemma~\ref{lma:ncl_cond}(\ref{lma:ncl_cond:ncl2}).
Now assume $u(\mathfrak{K}) = 1$.
Let $a_i$, $i\in\{0,\ldots,6\}$ be the spectrum of the $1$-classes in the factor plane, that is $a_i$ denotes the number of line classes incident with exactly $i$ $1$-classes.
Because of Lemma~\ref{lma:ncl_cond}(\ref{lma:ncl_cond:bs}), we get $a_6 = 0$.

Now we assume $a_0 = a_1 = a_2 = 0$.
So the standard equations on the spectrum are $a_3 + a_4 + a_5 = 31$, $3 a_3 + 4 a_4 + 5 a_5 = 6\#\mathfrak{K}$ and $3 a_3 + 6 a_4 + 10 a_5 = \binom{\#\mathfrak{K}}{2}$.
Solving this system of equations yields $a_4 = -\#\mathfrak{K}^2 + 43\#\mathfrak{K} - 465 = - \left(\#\mathfrak{K} - \frac{43}{2}\right)^2 - \frac{11}{4} < 0$, which is a contradiction.
This shows that there is a line class of multiplicity at most $2$.
\end{proof}

\subsection{Canonization in $\AHG(2,\Z_{25})$}
\label{subsec_canonization}
First we outline how to efficiently determine a canonical representative in $\AHG(2,\Z_{25})$ of a point set $\mathfrak{K}$ which contains a point triple in general position (no $2$ points in the same point class, no $3$ points in the same line class).\footnote{This is not a hard restriction: Every $2$-arc $\mathfrak{K}$ with $\#\mathfrak{K}\ge 7$ and $\#(\mathfrak{K}\cap[x])\le 2$ for all point classes $[x]$ contains a triple of points in general position.}
By the map mapping $(x,y)\mapsto x + 25y$, computed in the integers, we can compare points.
Sets of points are compared by the lexicographic ordering.
The smallest point set isomorphic to $\mathfrak{K}$ is called \emph{canonical} form of $\mathfrak{K}$.
Because $\mathfrak{K}$ contains a triple of points in general position, the canonical form of $\mathfrak{K}$ fulfills $\left\{\begin{pmatrix}0\\0\end{pmatrix},\begin{pmatrix}1\\0\end{pmatrix},\begin{pmatrix}0\\1\end{pmatrix}\right\}\subset\mathfrak{K}$ (the points correspond to the integers $0$, $1$, and $25$).

In order to efficiently canonize $\mathfrak{K}$ with respect to $\AGL(2,\Z_{25})$ we loop over all triples $\mathbf{k}_1,\mathbf{k}_2,\mathbf{k}_3\in\mathfrak{K}$ in general position, set $\mathbf{b}=-\mathbf{k}_1$, and uniquely determine $\mathbf{A}'$ via $\mathbf{A}'(\mathbf{k}_2-\mathbf{k}_1)=\begin{pmatrix}1\\0\end{pmatrix}$, $\mathbf{A}'(\mathbf{k}_3-\mathbf{k}_1)=\begin{pmatrix}0\\1\end{pmatrix}$, where we use the notation from subsection~\ref{subsect:collineations}.\footnote{The matrix $\mathbf{A}'$ does not exist in the cases where $\mathbf{k}_1$, $\mathbf{k}_2$ and $\mathbf{k}_3$ are not in general position.}
Thus we obtain a canonizer which needs at most $\#\mathfrak{K}^3$ operations.
We can easily extend this to a canonizer for $\PGL(3,\Z_{25})\down$ by looping over all $25$~possibilities for $\mathbf{c}$ and determining $\mathbf{A}'$ via $\mathbf{A}'(\mathbf{k}_2-\mathbf{k}_1)=\begin{pmatrix}1+\mathbf{c}(x,y)^T\\0\end{pmatrix}$, $\mathbf{A}'(\mathbf{k}_3-\mathbf{k}_1)=\begin{pmatrix}0\\1+\mathbf{c}(x,y)^T\end{pmatrix}$.
If we have an arc $\mathfrak{K}$ which is canonical with respect to $\PGL(3,\Z_{25})\down$ we can also revert these steps to obtain a complete list of $25$ arcs which are isomorphic to $\mathfrak{K}$ and canonical with respect to $\AGL(2,\Z_{25})$.

\subsection{Classification of arcs in $\AHG(2,\Z_{25})$ by orderly generation and integer linear programming}
The classification of arcs in $\AHG(2,\Z_{25})$ is explained at the example of the $(20,2)$-arcs.
We use an orderly generation approach (see \cite{0392.05001}), where we utilize integer linear programming to prune the search tree. Each $(20,2)$-arc $\mathfrak{K}$ in $\AHG(2,\Z_{25})$ corresponds to a solution of the  binary linear program (BLP)
\begin{eqnarray*}
  \sum\limits_{i\in\mathcal{P}(\AHG(2,\Z_{25}))} x_i\ge 20 &&\\
  \sum\limits_{i\in L} x_i \le 2&&\forall L\in\mathcal{L}(\AHG(2,\Z_{25}))\\
  x_i\in\{0,1\}&&\forall i\in\mathcal{P}(\AHG(2,\Z_{25})),
\end{eqnarray*}
via $\mathfrak{K}=\left\{i\in\mathcal{P}(\AHG(2,\Z_{25}))\mid x_i=1\right\}$.
Due to symmetry and the large integrality gap, this BLP can not be solved directly using customary ILP solvers like \texttt{ILOG CPLEX}.
To deal with the large automorphism group, one possibility would be to apply techniques e.~g.{} from \cite{1023.90088}.
We chose another possibility: By fixing some $x_i$ to one, i.~e.{} by prescribing some points of $\mathfrak{K}$, the biggest part of the symmetries is broken and we can try to solve the resulting BLP in reasonable time.
Of course for the search to be complete, all the non-isomorphic possibilities of prescribing some $x_i$ must be dealt with separately.

We can utilize this approach to deduce that for all point classes $[x]$ it holds $\#(\mathfrak{K}\cap [x])\in\{0,1\}$ as follows.
First we assume that $\mathfrak{K}$ contains some points $k_1,\dots,k_6$ with $\phi(k_i)=\phi(k_j)$ if and only if either $i=j$ or $\{i,j\}\in\Big\{\{1,2\},\{3,4\},\{5,6\}\Big\}$, i.~e.{} there are at least three $2$-classes
For each of the $104$, with respect to $\AGL(2,\Z_{25})$, non-isomorphic such subsets $\left\{k_1,\dots,k_6\right\}$ we check that the corresponding BLP is infeasible.
Thus by now we know that a $(20,2)$-arc in $\AHG(2,\Z_{25})$ can have at most two $2$-classes.

Next we prescribe points $k_1,\dots,k_5$, where  $\phi(k_i)=\phi(k_j)$ if and only if either $i=j$ or $\{i,j\}\in\Big\{\{1,2\},\{3,4\}\Big\}$. Here there are only $18$, with respect to $\AGL(2, \Z_{25})$, non-isomorphic possibilities. To forbid a third $2$-class we add the inequalities
\[
  \sum_{i\in \mathcal{C}} x_i\le 1\quad\quad\forall \mathcal{C}\in \Big\{[k] : k\in \mathcal{P}(\AHG(2,\Z_{25}))\Big\}\backslash\Big\{\left[k_1\right],\left[k_3\right]\Big\}
\]
to the BLP.
Again it turned out that there are no feasible solutions.
In the final step we we prescribe points $k_1,\dots,k_5$, where  $\phi(k_i)=\phi(k_j)$ if and only if either $i=j$ or $\{i,j\}=\{1,2\}$. In all $185$, with respect to $\AGL(2,\Z_{25})$, non-isomorphic cases the corresponding BLP is infeasible.

In the following we can assume $\#(\mathfrak{K}\cap [x])\in\{0,1\}$ for all point classes $[x]$.
As a new ingredient we utilize the concept of orderly generation, i.~e.{} if we prescribe a set $\mathfrak{K}'$ of points then we only search for arcs $\mathfrak{K}$, where $\mathfrak{K}$ is lexicographically minimal with respect to $\AGL(2,\Z_{25})$, i.~e.{} where $\mathfrak{K}\le \psi(\mathfrak{K})$ for all automorphisms $\psi\in\AGL(2,\Z_{25})$, and where the smallest $\#\mathfrak{K}'$ elements of $\mathfrak{K}$ are equal to $\mathfrak{K}'$. To this end we can formulate some necessary linear constraints: If for given $\mathfrak{K}'$ there are points $j\in \mathfrak{K}'$, $i\in\mathcal{P}(\AHG(2,\Z_{25}))$ and an automorphism $\psi$ such that $\psi\Big(\{i\}\cup \mathfrak{K}'\backslash\{j\}\Big)<\mathfrak{K}'$ then we can add the constraint
\begin{eqnarray*}
  \label{ie_orderly_point_constraint}
  x_i=0
\end{eqnarray*}
to our BLP. If there are points $j_1\neq j_2\in \mathfrak{K}'$, $i_1\neq i_2\in\mathcal{P}(\AHG(2,\Z_{25}))$ and an automorphism $\psi$ such that $\psi\Big(\{i_1,i_2\}\cup \mathfrak{K}'\backslash\{j_1,j_2\}\Big)<\mathfrak{K}'$ then we can add the constraint
\begin{eqnarray*}
  \label{ie_orderly_edge_constraint}
  x_{i_1}+x_{i_2}\le 1
\end{eqnarray*}
to the BLP.
Using these constraints we were able to drastically reduce the number of lexicographically smallest sets $\mathfrak{K}'=\left\{k_1,\dots,k_5\right\}$ which possibly can be extended to $(20,2)$-arcs in $\AHG(2,\Z_{25})$ using the feasibility of the BLP. For each such \textit{partial} arc $\mathfrak{K}'$ we can easily determine a set $\mathfrak{K}''$ of points such that we have $\mathfrak{K}\cap\mathfrak{K}''=\emptyset$ for all lexicographically minimal $(20,2)$-arcs $\mathfrak{K}$ with $\mathfrak{K}'\subset\mathfrak{K}$. (This set consists of the $i\in\mathcal{P}(\AHG(2,\Z_{25}))$ where we would add $x_i=0$ to the BLP.)

After these preparative calculations we have classified the $(20,2)$-arcs in $\AHG(2,\Z_{25})$ using a branch\&bound approach as follows. We start with one of the $\mathfrak{K}'$, where the points from the corresponding sets $\mathfrak{K}''$ are forbidden, as described above, and extend the partial arcs point by point by branching on the point classes $[x]$. (See \cite{0297.68037} for the question where to branch.) I.~e.{} for a given point class $[x]$ we loop over all possibilities to extend the current partial arc with a point from $[x]$ or to take none of the points from $[x]$.

In the bounding step we perform two checks. If the cardinality of the partial arc $\mathfrak{K}$ is either at most $8$ or $20$ we check whether there is an automorphism $\psi\in\AGL(2,\Z_{25})$ such that the smallest five elements of $\psi(\mathfrak{K})$ are lexicographically smaller than the initial $\mathfrak{K}'$, see Subsection~\ref{subsec_canonization} for the details. If such a $\psi$ exists, we can prune the search tree (isomorphism-pruning). For the other check we count the number $e$ of point classes $[x]$, where there is a point $k\in[x]\backslash\left(\mathfrak{K}\cup\mathfrak{K}''\right)$ such that $\mathfrak{K}\cup\{k\}$ is a $2$-arc. If $\#\mathfrak{K}+e<20$ we can prune the search tree. (The latter bound relies on the fact that we only have to consider arcs without $2$-classes.)

The maximum size of a $2$-arc in $\AHG(2,R)$ is $20$, this was determined computationally in \cite{Kurz-2009-SJoC3:159-178}.
Now we can give the number of such $(20,2)$ arcs up to different automorphism groups:
In Table~\ref{table_affine_20er} we give the number of non-isomorphic $(20,2)$-arcs in $\AHG(2,\Z_{25})$ per size of its stabilizer $\Aut_1$ in $\AGL(2,\Z_{25})$ (collineations of $\AHG(2,\Z_{25})$ preserving the standard parallelism), $\Aut_2$ in $\PGL(3,\Z_{25})\down$ (all collineations of $\AHG(2,\Z_{25})$) and $\Aut_3$ in $\PGL(3,\Z_{25})$ (all collineations of $\PHG(2,\Z_{25})$, via the standard embedding). We verified that none of these arcs can be extended to a $(22,2)$-arc in $\PHG(2,\Z_{25})$.

In fact, all these $(20,2)$-arcs are maximal in $\PHG(2,\Z_{25})$.
This is consistent with our main result that the $(21,2)$-arc in $\PHG(2,\Z_{25})$ is unique, because all line classes of the $(21,2)$-arc contain at least two points of this arc.
 
\begin{table}[htp]
  \begin{center}
    \begin{tabular}{rrrrr}
      \hline
      $\# \Aut_\star$ & 1 & 2 & 4 & $\Sigma$\\
      \hline
      $\# \Aut_1$ & 11198 & 226 & 6 & 11430\\
      $\# \Aut_2$ & 420 & 62 & 6 & 488\\
      $\# \Aut_3$ & 415 & 47 & 4 & 466\\
      \hline
    \end{tabular}
    \caption{Number of non-isomorphic $(20,2)$-arcs in $\AHG(2,\Z_{25})$ per size of the automorphism group.}
    \label{table_affine_20er}
  \end{center}
\end{table}

In the same way, the $(19,2)$-arcs in $\AHG(2,\Z_{25})$ were generated up to isomorphism.
Interestingly enough, it turns out that a $(19,2)$-arc in $\AHG(2,\Z_{25})$ has no or exactly four $2$-classes.
The results of the former case are summarized in Table~\ref{table_affine_19er_no_doubly_nk}.
In the latter case there are exactly $75$, $3$, $3$ non-isomorphic (with respect to the different automorphism groups) solutions, each with trivial stabilizer in $\PGL(3,\Z_{25})$.

\begin{table}[htp]
  \begin{center}
    \begin{tabular}{rrrrr}
      \hline
      $\#\Aut_\star$ & 1 & 2 & 3 & $\Sigma$\\
      \hline
      $\#\Aut_1$ & 3922099 & 2555 & 33 & 3924687 \\
      $\#\Aut_2$ & 156669 & 511 & 33 & 157213\\
      $\#\Aut_3$ & 146610 & 511 & 33 & 147154\\
      \hline
    \end{tabular}
    \caption{Number of non-isomorphic $(19,2)$-arcs in $\AHG(2,\Z_{25})$ with maximum point class multiplicity $1$.}
    \label{table_affine_19er_no_doubly_nk}
  \end{center}
\end{table}

A check for extendability revealed that among the $157213 + 3$ $\PGL(3,\Z_{25})\down$-representatives of $(19,2)$-arcs in $\AHG(2,\Z_{25})$ there is only a single arc $\mathfrak{K}$ which is extendable to a $(21,2)$-arc in $\PHG(2,\Z_{25})$.
Furthermore, the extension of $\mathfrak{K}$ is unique up to isomorphism, so together with Lemma~\ref{lma:projToAff} this proves Theorem~\ref{thm:main}.

More precisely, $\mathfrak{K}$ is given by
\begin{multline*}
\{
(0,0),
(0,1),
(15,12),
(20,3),
(15,4),
(1,0),
(1,1),
(6,12),
(21,9),
(2,15),\\
(12,3),
(3,11),
(3,2),
(8,14),
(9,10),
(14,11),
(4,2),
(4,13),
(14,19)
\}.
\end{multline*}
By the standard embedding in $\PHG(2,\Z_{25})$, $\mathfrak{K}$ can be extended by the points
\[(20:1:5)\quad\mbox{and}\quad(5:1:7).\]

The uniqueness of $\mathfrak{K}$ necessarily means that that all affine subsets of size $19$ of the unique $(21,2)$-arc $\mathfrak{k}$ in $\PHG(2,\Z_{25})$ are isomorphic.
Again this is consistent with the analysis of $\mathfrak{k}$ given in \cite{Kiermaier-Koch-2009-ProcOC11:106-113}: $\mathfrak{k}$ has maximum point class multiplicity $u(\mathfrak{k}) = 1$, and in the factor plane the $0$-classes form a projective triangle $\Delta$ extended by the unique fixed point under the action of the stabilizer of $\Delta$ on the factor plane.
Therefore, there are $3$ possibilities for the choice of a line class at infinity, namely the $3$ edges of the projective triangle.
Under the stabilizer of $\mathfrak{k}$ which has order $3$, these three line classes are in the same orbit.

\section{Further results}
\label{sect:further}
\subsection{Number of non-isomorphic $(n,2)$-arcs for large $n$}
We used the same methods to compute further numbers of isomorphism types of $2$-arcs.
Wherever possible, both computational approaches were used to assure the correctness of the result.

Table~\ref{table_projective_arcs} shows the number of $\PGL(3,\Z_{25})$-isomorphism types of $(n,2)$-arcs in $\PHG(2,\Z_{25})$.
The column $s_i$ lists the number of $(n,2)$-arcs whose stabilizer has size $i$, and the column $\Sigma$ lists the numbers of all $(n,2)$-arcs.
In the column ''time'', the running time of the algorithm described in Section~\ref{sect:matthias} is given.

Remarkably, the stabilizers of large $2$-arcs in $\PHG(2,\Z_{25})$ are relatively small.
For comparison we mention that in $\PHG(2,\IS_5)$ there exists a $(25,2)$-arc whose stabilizer has order $300$.
This implies that using the method of prescribed automorphisms like in \cite{Kiermaier-Kohnert-2007-ProcOC10:112-119}, a $2$-arc of size at least $18$ in $\PHG(2,\Z_{25})$ is much harder to find than a $(25,2)$-arc in $\PHG(2,\IS_5)$.

\begin{table}[htp]
  \begin{center}
    \begin{tabular}{rrrrrrrrrrr}
      \hline
      $n$  &    $s_1$   & $s_2$   & $s_3$ & $s_4$ & $s_6$ & $s_8$ & $s_{12}$ & $s_{24}$ &   $\Sigma$ &       time \\
      \hline
      $22$ &            &         &       &       &       &       &          &          &        $0$ &   $30002$s \\
      $21$ &            &         &   $1$ &       &       &       &          &          &        $1$ &   $47750$s \\
      $20$ &      $591$ &  $63$   &       &   $4$ &       &       &          &          &      $658$ &  $155258$s \\
      $19$ &   $221374$ & $687$   &  $71$ &       &       &       &          &          &   $222132$ &  $487340$s \\
      $18$ & $23880140$ & $16842$ & $721$ & $271$ &  $73$ &  $11$ &      $9$ &      $4$ & $23897599$ & $2299265$s \\
      \hline
    \end{tabular}
    \caption{Number of non-isomorphic $(n,2)$-arcs in $\PHG(2,\Z_{25})$.}
    \label{table_projective_arcs}
  \end{center}
\end{table}

\subsection{The maximum size of a $2$-arc of given maximum point class multiplicity}
Using a combination of the computational results and geometric reasoning, we are ready to determine the maximum size of a $2$-arc $\mathfrak{K}$ in $\PHG(2,\Z_{25})$ and $\PHG(2,\IS_5)$ of a given maximum point class multiplicity $u(\mathfrak{K})$:

\begin{theorem}\label{thm:u}
Let $u\in\{1,\ldots,q+1\}$ and $m_{2,u}(R)$ be the maximum size of a $2$-arc $\mathfrak{K}$ in $\PHG(2,R)$ with $u(\mathfrak{K}) = u$.

For $R \in\{\Z_{25},\IS_5\}$ the values of $m_{2,u}(R)$ are:
\[
\begin{array}{c|cccccc}
   & \multicolumn{6}{|c}{u} \\
   \hline
   R & 1 & 2 & 3 & 4 & 5 & 6\\
   \hline
   \Z_{25} & 22 & 19 & 12 & 12 & 10 & 6\\
   \IS_5 & 25 & 22 & 12 & 12 & 10 & 6
\end{array}
\]
\end{theorem}

\begin{proof}
Let $P_u$ be the set of all $u$-classes and $x$ be a $u$-class.

For a line class $[L]$, the non-increasing sequence of the multiplicities of the point classes incident with $[L]$ is called \emph{type} of $[L]$.
The type of $[L]$ is a partition of the multiplicity of $[L]$ into $6$ summands.
In the following, we assume the usual partial order on the set of partitions.

By Lemma~\ref{lma:ncl_cond}(\ref{lma:ncl_cond:q+1}), types greater or equal to $(2,1,1,1,1,0)$ or greater or equal to $(1,1,1,1,1,1)$ are impossible, and by the argument in Remark~\ref{rmk:ubig}, types greater than $(2,2,1,1,0,0)$ or greater than $(2,2,2,0,0,0)$ are impossible.

Furthermore, types greater or equal to $(3,3,1,0,0,0)$ are impossible:
Assume that there is a line class $[L]$ of type $(3,3,1,0,0,0)$.
Let $x$ be the point of $\mathfrak{K}$ corresponding to the entry $1$ in the type.
All lines of $\Pi_{[L]}$ incident with $x$ contain at most $2$ points of the induced point set $\mathfrak{k}$, and the line incident with $x$ and $p_\infty$ intersects $\mathfrak{k}$ only in $x$.
This gives $\#\mathfrak{k} \leq 6$, a contradiction.

For $u = 6$, it is clear that the points of $\mathfrak{K}$ in $[x]$ must form an oval, and no further point can be added to $\mathfrak{K}$.

For $u = 5$, the points in $[x]$ determine all but possibly one line class $[L]$ incident with $[x]$.
To get a $2$-arc $\mathfrak{K}$, at most one further point class $[y]$ can have non-empty intersection with $\mathfrak{K}$, and $[y]$ must be incident with $[L]$.
So $\#\mathfrak{K} \leq 10$.
If we arrange the points of $\mathfrak{K}$ within $[x]$ and $[y]$ such that they do not determine $[L]$, we get indeed a $(10,2)$-arc, showing $m_{2,5}(R) = 10$.

For $u = 4$, each $4$-class $[y]$ is incident with at most $2$ line classes not determined by $[y]$.
By the restrictions on the types, no three $4$-classes are collinear.
So $\#P_4\leq 3$.
For $\#P_4 = 3$, all point classes not in $P_4$ are $0$-classes.
Furthermore, it is easy to see that there is indeed a $(12,2)$-arc of this structure.
For $\#P_4 = 2$, let $P_4 = \{[x],[y]\}$.
The line class $[L]$ incident with $[x]$ and $[y]$ must be of type $(4,4,0,0,0,0)$.
Outside of $[L]$ there is at most one point class of non-zero multiplicity, so $\#\mathfrak{K}\leq 2\cdot 4 + 3 = 11$.
For $\#P_4 = 1$ the restrictions on the types show that a line class incident with $[x]$ has multiplicity at most $7$, which gives $\#\mathfrak{K} \leq 4 + 2\cdot (7 - 4) = 10$.

Now assume $u = 3$.
Again, no three $3$-classes are collinear.
Each $3$-class is incident with at least three line classes of multiplicity $3$.
So if there are three non-collinear $3$-classes, there is at most one further point class which is not a $0$-class.
This leads to $\#\mathfrak{K}\leq 4\cdot 3 = 12$, and it is straightforward to check that there is a $(12,2)$-arc of this structure.
For $\#P_3 = 2$, the line class $[L]$ incident with the two $3$-classes must be of type $(3,3,0,0,0,0)$.
Outside of $[L]$ there remain at most $4$ point classes $[x_i]$ which are not a $0$-class.
There are $4$ line classes through a $3$-class and two distinct point classes $[x_i]$, $[x_j]$.
Since the type $(3,2,2,0,0,0)$ is not possible, at most $2$ of the point classes $[x_i]$ can be $2$-classes.
In total, $\mathfrak{K} \leq 2\cdot 3 + 2\cdot 2 + 2\cdot 1 = 12$.
In the case $\#P_3 = 1$, by the restrictions on the types the line classes incident with $[x]$ have multiplicity at most $6$, so again $\#\mathfrak{K}\leq 3 + 3\cdot (6 - 3) = 12$.

The values $m_{2,u}(\Z_{25})$ for $u\in\{1,2\}$ follow from the computer classification in this article.
$m_{2,2}(\IS_{5}) = 22$ follows from the existence of the $(22,2)$-arc over $\IS_5$ given in \cite{Kiermaier-Koch-2009-ProcOC11:106-113} and Lemma~\ref{lma:ncl_cond}(\ref{lma:ncl_cond:u2}).
Finally, the value $m_{2,1}(\IS_5) = 25$ follows from \cite{Honold-Kiermaier-maximal2Arcs-2010}.
\end{proof}

\bibliographystyle{amsplain}

\providecommand{\bysame}{\leavevmode\hbox to3em{\hrulefill}\thinspace}
\providecommand{\MR}{\relax\ifhmode\unskip\space\fi MR }
\providecommand{\MRhref}[2]{%
  \href{http://www.ams.org/mathscinet-getitem?mr=#1}{#2}
}
\providecommand{\href}[2]{#2}

\medskip
 {\it E-mail address: }michael.kiermaier@uni-bayreuth.de\\
 \indent{\it E-mail address: }matthias.koch@uni-bayreuth.de\\
  \indent{\it E-mail address: }sascha.kurz@uni-bayreuth.de\\
\end{document}